\documentclass[a4paper]{article}

\usepackage{amsfonts,amsmath,amsthm,amssymb}
\usepackage{natbib}

\RequirePackage{xcolor}
\definecolor[named]{ACMBlue}{cmyk}{1,0.1,0,0.1}
\definecolor[named]{ACMYellow}{cmyk}{0,0.16,1,0}
\definecolor[named]{ACMOrange}{cmyk}{0,0.42,1,0.01}
\definecolor[named]{ACMRed}{cmyk}{0,0.90,0.86,0}
\definecolor[named]{ACMLightBlue}{cmyk}{0.49,0.01,0,0}
\definecolor[named]{ACMGreen}{cmyk}{0.20,0,1,0.19}
\definecolor[named]{ACMPurple}{cmyk}{0.55,1,0,0.15}
\definecolor[named]{ACMDarkBlue}{cmyk}{1,0.58,0,0.21}

\newtheorem{theorem}{Theorem}[section]
\newtheorem{corollary}{Corollary}[theorem]
\newtheorem{lemma}[theorem]{Lemma}
\newtheorem{proposition}[theorem]{Proposition}
\theoremstyle{definition}
\newtheorem{definition}[theorem]{Definition}
\theoremstyle{remark}
\newtheorem{example}[theorem]{Example}

\usepackage{hyperref}
\hypersetup{colorlinks,
    linkcolor=ACMBlue,
    citecolor=ACMPurple,
    urlcolor=ACMLightBlue,
    filecolor=ACMDarkBlue}

\hypersetup{pdflang={en}, pdfdisplaydoctitle}

\newcommand{\sset}[1]{\ensuremath{\mathbb{#1}}} 
\newcommand{\zz}{\sset{Z}} 
\newcommand{\nn}{\sset{N}} 
\newcommand{\cc}{\sset{C}} 

\newcommand{\cA}{\mathcal{A}}
\newcommand{\cB}{\mathcal{B}}
\newcommand{\cD}{\mathcal{D}}
\newcommand{\cG}{\mathcal{G}}
\newcommand{\cM}{\mathcal{M}}
\newcommand{\cO}{\mathcal{O}}
\newcommand{\dres}{\partial{\rm Res}}

\DeclareMathOperator{\ord}{ord} 
\DeclareMathOperator{\lc}{lc} 

\newcommand{\kk}{\sset{K}} 

\newcommand{\Const}{\mathbf{C}} 
\newcommand{\centr}{\mathcal{Z}} 
\DeclareMathOperator{\bc}{BC} 
\def\Spec{{\rm Spec}}

\newcommand{\<}{\overline{<}}

\usepackage[linesnumbered,ruled,vlined]{algorithm2e}
\SetKwRepeat{Do}{do}{while}
\SetKw{True}{\texttt{True}}
\SetKw{False}{\texttt{False}}
\SetKwInOut{Input}{Input}
\SetKwInOut{Output}{Output}
\SetKwComment{Comment}{// }{}

\title{Gröbner bases of Burchnall-Chaundy ideals for ordinary differential operators}
\author{Antonio Jiménez-Pastor\\Universidad Politécnica de Madrid\\ Applied Mathematics Department, ETSAM\\Madrid, Spain \and Sonia L. Rueda\\Universidad Politécnica de Madrid\\ Applied Mathematics Department, ETSAM\\Madrid, Spain
} 

\begin{document}

\maketitle

\begin{abstract}
The correspondence between commutative rings of ordinary differential operators (ODOs) and algebraic curves was established by Burchnall and Chaundy, Krichever and Mumford, among many others. To make this correspondence computationally effective, in this paper we aim to compute the defining ideals of spectral curves, Burchnall-Chaundy (BC) ideals. We provide an algorithm to compute a Gröbner basis of a BC ideal. The point of departure is the computation of the finite set of generators of a maximal commutative ring of ODOs, which was implemented by the authors in the package \texttt{dalgebra} of SageMath. The algorithm to compute BC ideals has been also implemented in \texttt{dalgebra}. 

 The differential Galois theory of the corresponding spectral problems, linear differential equations with parameters, would benefit from the computation on this prime ideal, generated by constant coefficient polynomials. In particular, we prove the primality of the differential ideal generated by a BC ideal, after extending the coefficient field. This is a fundamental result to develop Picard-Vessiot theory for spectral problems.

\end{abstract}

\section{Introduction}\label{sec:introduction}

This paper is devoted to the computation of the {\it spectral curve} $\Spec (\cA)$ of a commutative domain of ordinary differential operators $\cA$. In particular, given an ordinary differential operator $L$ with coefficients in a differential field $(\kk,\partial)$, the centralizer $\centr(L)$ of $L$ in the ring of differential operators $\cD=\kk [\partial]$ is a maximal commutative subring of $\cD$ and, in particular, a commutative domain. We will compute the defining ideal of the spectral curve $$\Gamma=\Spec(\centr(L)).$$

These algebraic curves are tied to the theory of commuting differential operators since the visionary work of Burchnall and Chaundy~\cite{BC1}. The classification of all commutative algebras of ODOs has been studied by many authors including Krichever~\cite{K77bis,K78} and Mumford~\cite{Mum2}. Important works~\cite{BZ,GD,Mulase,PW,Rueda2025,Schilling,SW,Ve}, contributed to a bijective correspondence between the set of equivalence classes of commutative algebras $\cA$ of ordinary differential operators with a monic element, (the algebras $\cA$ and $\sigma \cA \sigma^{-1}$ are identified for an invertible function $\sigma$), and the set of equivalence classes of isomorphism classes of {\it quintets} consisting of an algebraic curve and on it a point, a vector bundle, a local covering and a local trivialization.

Spectral curves are used to produce solutions of non-linear partial differential equations~\cite{K77bis,K78,Wilson1985},  solutions of the Gelfand-Dickey (GD) hierarchies~\cite{GD,JimenezPastor2025a}. On the other hand, spectral curves are fundamental objects in studying the Picard-Vessiot theory of spectral problems $(L-\lambda)(y)=0$, for a generic parameter $\lambda$ and a differential operator $L$, whose coefficients are solutions of one of the equations of the GD hierarchies~\cite{BEG,MRZ2,MRZ1,Rueda2024}. This condition is equivalent to having a non-trivial centralizer $\centr(L)$, i.e., to containing an operator $A$ which is not a polynomial in $L$ with constant coefficients. 

Given the field of constants $\Const$ of $\kk$, the centralizer of $L$ contains $\Const[L,A]$ whose spectral curve is defined by $f(\lambda,\mu)=0$, for a constant coefficient polynomial $f$ that can be computed as the square free part of the differential resultant $\dres(L-\lambda, A-\mu)$. The spectral curve of $\Const[L,A]$ is the algebraic curve $\Gamma_{L,A}$ whose defining ideal is the prime ideal $(f)$, see~\cite{Rueda2025,Rueda2024,Zheglov}, and $\Const[L,A]$ is isomorphic to the coordinate ring 
\begin{equation}
    \Const[\Gamma_{L,A}]:=\frac{\Const[\lambda,\mu]}{(f)}.
\end{equation}
Moreover $f$ is a Burchnall-Chaundy polynomial, since $f(L,A)=0$.  

In general, commutative domains in $\cD$ have more than two generators. For instance, the generic situation for ODOs of order 3 in $\cD$ will be that the centralizer $\centr (L)=C[L,A_1,A_2]$, having 3 generators, that reduce to two generators only in very specific cases~\cite{Rueda2025,Rueda2024}. The defining ideal of the spectral curve of the centralizer was computed in~\cite{Rueda2024} using differential resultants but we will not use differential resultants in this paper, since this result does not generalize to operators of order greater than $3$.

The centralizer $\centr(L)$ is a $\Const[L]$-module generated by a finite basis $\{1$, $G_1$,$\ldots,$ $G_{t-1}\}$ \cite{Goodearl1983}, that provides a set of generators as a $\Const$-algebra
\begin{equation}
    \centr(L)=\Const[G_0:=L,G_1,\ldots ,G_{t-1}].
\end{equation}
These facts are reviewed in Section \ref{sec:centralizer} and will allow us to describe the spectral curve $ \Gamma$ algorithmically, controlling the defining ideal of this (in principle) abstract algebraic curve.
In this paper we compute a Gröbner basis of the defining  ideal of $\Gamma$ and prove that it is the ideal of all Burchnall-Chaundy polynomials, those vanishing in the set of generators of $\centr (L)$. As in~\cite{Rueda2025,Rueda2024} we call it the Burchnall-Chaundy (BC) ideal of $L$ and denote it $\bc (L)$, establishing an explicit isomorphism between $\centr (L)$ and the coordinate ring of the spectral curve 
\begin{equation}
    \Const [\Gamma]:=\frac{\Const [\lambda, \mu_1,\ldots ,\mu_{t-1}]}{\bc (L)}.
\end{equation}
In~\cite{JimenezPastor2025b} we implemented an algorithm, in the \texttt{dalgebra} package of SageMath~\cite{Zenodo2025}, that allows to compute a $\Const[L]$-basis of the centralizer $\centr(L)$. In this paper, we present an algorithm to compute the algebraic relations between the $\Const[L]$-module generators. The algorithm has also been implemented in \texttt{dalgebra} and ultimately computes a Gröbner basis of $\bc(L)$. We define BC ideals and spectral curves in Section \ref{subsec:bc}. The main results of this paper are included in Section \ref{sec:result}.

Since $\centr (L)\simeq \Const [\Gamma]$ is a domain, the BC ideal is a prime ideal in $\Const [\lambda, \overline{\mu}]$. 
A Picard-Vessiot theory for the spectral problem $(L-\lambda)(y)=0$ needs a coefficient field containing $\kk$ and $C[\Gamma]$~\cite{BEG,MRZ2,Rueda2025}. To allow effective calculations, one considers the differential ring  
\begin{equation}
    \kk [\Gamma]:=\frac{\kk [\lambda, \mu_1,\ldots ,\mu_{t-1}]}{[\bc (L)]}.
\end{equation}
As a first application of the Gr\" obner basis of $\bc (L)$ we prove that the differential ideal $[\bc (L)]$ generated in $\kk [\lambda, \mu_1,\ldots ,\mu_{t-1}]$ by $\bc (L)$ is also a prime ideal in Section \ref{sec:primality}. This result allows to consider $\kk (\Gamma)$, the fraction field of the domain $\kk [\Gamma]$ to develop this Picard-Vessiot theory. Conclusions and further research are included in Section \ref{sec:problems}.

\section{Basic definitions}\label{sec:definitions}

This section focuses on basic definitions and classical results that will be necessary to understand this paper, 
since  this work mixes the world of differential operators, and their centralizers, with algebraic geometry and commutative algebra.


    {\bf Commutative algebra.} Let us begin with some concepts from commutative algebra, for further references we refer to~\cite{Atiyah1969,Bourbaki1998,Lang2002}. 
    Let $\Const$ be a field  and consider a finitely generated $\Const$-algebra $\cA$. 
    There exists $a_1,\ldots,a_n \in \cA$ such that every element in $\cA$ can be written as a polynomial in $a_1,\ldots,a_n$ with coefficients in $\Const$. These finitely generated algebras are domains if they are isomorphic to quotients of a polynomial ring by a prime ideal. If the generators of $\cA$ are known, the isomorphism $\phi$ can be made explicit, as an evaluation map sending  an algebraic variable $y_i$ to each generator $a_i$ so that $\frac{\Const[y_1,\ldots,y_n]}{\ker(\phi)}$ is isomorphic to $\cA=\Const[a_1, \ldots,a_n ]$. The prime ideal $\ker(\phi)$ describes the algebraic relations between these generators of $\cA$.
    
    As it is well known,  Gröbner bases are one of the main tools of symbolic computation to handle polynomial ideals. We refer to~\cite{Buchberger1998,Cox1996} for definitions and basic results on monomial orders but we include next the characterization of Gröbner basis that will be used in this paper. 
    
    Let us fix a monomial order $<$ in the polynomial ring $\Const[Y]$ generated by the set of variables $Y=\{y_1,\ldots ,y_n\}$. Given $p\in \Const[Y]$, let us denote by $\deg_Y(p)$ the total degree of $p$ w.r.t. the variables in $Y$, by
    $\ell (p)$ its leading monomial w.r.t. the monomial ordering $<$, and by $\lc(p)$ its leading coefficient. Given a subset $S$ of polynomials in $\Const[Y]$, denote by $\ell(S)$ the ideal in $\Const [Y]$ generated by the leading monomials of $S$.

    Let $\mathcal{B} = \{f_1,\ldots, f_m\}$ be a set of generators for an ideal $I \subseteq \Const[Y]$. The \emph{leading ideal} of $I$ is $\ell (I)$.
    
    \begin{definition}[{\cite[Def. 2.5.5]{Cox1996}}]\label{def:groebner}
        The set $\mathcal{B} $ is a \emph{Gröbner basis} of the ideal $I=(\cB)$ if the leading ideal of $I$ is generated by the leading monomials of $\mathcal{B}$, that is $\ell(I)=\ell(\cB)$. 
    \end{definition}

    By the Division Theorem (see \cite{Cox1996}, Theorem 3, page 64), we can write $p\in \Const [Y]$ as $p =q +\tilde{p}$, where $q\in (\cB)$ and $\tilde{p}$ is the normal form of $p$ w.r.t. $\cB$.
    It is well known~\citep{Buchberger1998,Cox1996} that  $\mathcal{B}$ being a Gröbner basis is equivalent to a polynomial $p\in \Const[Y]$ reducing to $0$ w.r.t $\mathcal{B}$, $\tilde{p}=0$, if and only if $p \in I$.


    {\bf Differential fields.} We refer to~\cite{Kolchin1973,Ritt1950} for further references to basic facts on differential algebra. Let $(\kk,\partial)$ be a differential field, i.e., a field $\kk$ with a derivation $\partial$, an additive homomorphism $\partial: \kk \rightarrow \kk$ that satisfies Leibniz rule ($\partial(ab) = \partial(a)b + a\partial(b)$). The set of constants, elements $c\in \kk$ with $\partial(c) = 0$, is a subfield $\Const$ of $\kk$, its field of constants.Throughout this paper, we will assume that $\Const$ is algebraically closed and has zero characteristic.

    A differential extension $\mathbb{F} \supset \kk$ is a field extension where the derivation of $\mathbb{F}$ extends the derivation over $\kk$. In these extensions it is clear that the field of constants $\Const$ may be also extended. For many algorithms the field of constants are necessary, so these extensions are built keeping the same field of constants. For example, strongly normal extensions add the solutions to a non-linear differential equations while preserving the constants (see~\cite[Section VI.3]{Kolchin1973} or~\cite{Kumbhakar2024,Kumbhakar2025}), or monomial extensions add a new transcendental element with prescribed derivation with conditions for preserving the constants~\cite{Bronstein1997}.

    Throughout this paper, we will use the next differential fields to illustrate our results with examples: the monomial extension $\cc(e^x)$ of the field of complex numbers $\cc$ and the strongly normal extension $\cc\langle \wp \rangle=\cc(\wp,\wp')$ defined by the Weierstrass $\wp$ function $\wp=\wp(x;g_2,g_3)$ where $\wp'^2 = 4\wp^3 - g_2\wp - g_3$ for some constants $g_2,g_3\in \cc$ such that $g_2^3-27g_3^2\neq0$. In both cases the derivation is $\partial=d/dx$.

{\bf Differential operators.} 
The ring of differential operators $\kk [\partial]$ is a (left and
right) non commutative Euclidean domain, see for instance
\cite{BGV}. A differential operator $P\in \kk [\partial]$ can be written in a unique way as
\[P = a_n \partial^n + a_{n-1}\partial^{n-1}+\ldots +a_1\partial+a_0, \mbox{ with } a_i\in \kk,\,\,\,a_n\neq 0.\]
The \emph{order} of $P$ is $n$, denoted  $\ord(P)$,  with the convention $\ord (0)=-\infty$. The coefficient of the highest order term of~$P$ is the \emph{leading coefficient} of $P$, denoted $\lc(P)=a_n$. We say that an operator~$P\in \kk [\partial]$ of order $n$ is in \emph{normal form} if $\lc(P)=1$ and the coefficient of the term of order $n-1$ is zero.
In  \cite{Mu2, Zheglov} these operators are called {\it elliptic}. Differential operators can always be brought to normal form, by using a change of variable and conjugation by a function.

    \section{Centralizers}\label{sec:centralizer}

    Consider a differential operator $L \in \kk[\partial]$ of order $n$ and in normal form
    \[L=\partial^n + u_{n-2}\partial^{n-2} + \ldots + u_1\partial + u_0.\]
The centralizer of $L$ in $\Sigma [\partial]$ is 
\[\centr (L)=\{A\in \Sigma[\partial]\mid [L,A]=0\}.\]

 It is well known \cite{Wilson1985} that $L$ has a non-trivial centralizer, i.e., $\centr(L) \neq \Const[L]$, if the coefficients $u_2,\ldots,u_n$ satisfy the equations of one of the systems of the Gelfand-Dickey hierarchies, which for $n=2$ is the Korteweg-de Vries hierarchy and for $n=3$ is the Boussinesq hierarchy. In~\cite[Section~5]{JimenezPastor2025a}, the authors designed an algorithm to explicitly compute the equations of these hierarchies, saving the results in a repository for a wide range of values for $n$. In ~\cite[Section~8]{JimenezPastor2025b}, a method was designed to compute specific solutions $u_0,u_1,\ldots ,u_{n-2}$ of GD-hierarchies in a chosen strongly normal extension $\Const \langle \eta \rangle$, transforming systems of non-linear differential equations into systems of polynomial equations. The examples of ODOs with non-trivial centralizer used in this paper were computed using these techniques.
 
 In this paper, we will assume that an operator $L$ with a non-trivial centralizer is given. Moreover, we will assume that a set of generators of $\centr (L)$ is already known. In~\cite{JimenezPastor2025b}, the authors proposed an algorithm to compute $\centr(L)$. The starting point of this algorithm is the certainty of having a non-trivial centralizer, in other words, to be given the order of a differential operator in $\centr(L)\setminus \Const [L]$. In this situation, one can always compute the \emph{level} of $\centr(L)$, that is, the minimal order $M$ of the elements in $\centr(L)\setminus \Const[L]$. Given an operator $L$ of $\kk [\partial]$ in normal form and the level $M$, the algorithm computes a basis of $\centr(L)$ as a $\Const[L]$-module. 
 The algorithm in \cite{JimenezPastor2025b} is based on~\cite[Theorem 1.2]{Goodearl1983}, that we review next. See also~\cite[Theorem~3.7]{JimenezPastor2025b} for further details. Let us denote by $\zz_n$ the cyclic group of classes $(\text{mod } {n})$ and by $[m]_n$ the class of $m\in \zz \pmod{n}$. 
    
    \begin{theorem}[Goodearl]\label{thm:goodearl}
      Let $L \in \kk[\partial]$ be a differential operator of order $n$. Then a basis of $\centr (L)$ as a $\Const[L]$-module consists of $t$ operators 
      \[\mathcal{G} = \{G_0 = 1,G_1,\ldots,G_{t-1}\},\]
      satisfying the following conditions:
      \begin{enumerate}
        \item Each $G_k$ has minimal order  within the set $\{Q\in \centr (L)\mid \ord(Q)\equiv \ord(G_k) \pmod{n}\}$, we say $G_k$ is \emph{order minimal}.
        \item The orders $(\text{mod } {n})$ of all elements of $\centr (L)$ form an additive subgroup of $\ \zz_n$ of cardinal $t$
\begin{equation}
    \mathcal{O} = \{[\ord(G_k)]_n \ :\ k = 0,\ldots,t-1\}. 
\end{equation}
    \end{enumerate}
Therefore, $t=|\cO|$ divides $n$.
    \end{theorem}

The \emph{rank} of a commuting set of ODOs is the greatest common divisor of all the orders. The following consequence of the previous theorem was proved in \cite{JimenezPastor2025b}.

\begin{corollary}
Let $L \in \kk[\partial]$ be a differential operator of order $n$ with group of orders $\cO$ of cardinal $t$. Then the rank of $\centr (L)$ equals $n/t$. 
\end{corollary}

If the rank of $\centr (L)$ is $1$ then $L$ is said to be \emph{algebro-geometric}. In this case there exists an operator in $\centr (L)\backslash \Const [L]$ of order coprime with $n$. Furthermore, in the rank one case $t=n$.

Centralizers are maximal commutative subrings of $\kk [\partial]$.
To each cyclic subgroup $\cO'$ of $\cO$ we can associate $\cA_{\cO'}$, a $\Const [L]$-submodule of $\centr (L)$, with basis the subset $\cG'\subset \cG$ of those $G_k$ such that $[\ord(G_k)]_n\in \cO'$. Observe that $\cA_{\cO'}$ is a maximal commutative subalgebra of $\centr (L)$ of rank $n/t'$, with $t'=|\cO'|$.
    
The output of the algorithm in \cite{JimenezPastor2025b} is the basis as a $\Const[L]$-module of a maximal commutative subalgebra $\cA_{\cO'}$ of $\centr(L)$ of rank dividing $(n,M)$. When $n$ and $M$ are not coprime, the algorithm is not guaranteed to compute a basis of the whole centralizer.
We will not address this limitation, and we will assume that we have computed a basis $\mathcal{G} = \{G_0,\ldots,G_{t-1}\}$ of $\centr(L)$ as a $\Const[L]$-module. However, we remark that results on this paper (see Section~\ref{sec:result}) can still be applied to the case where we do not have the full centralizer adapting the statements to the corresponding maximal commutative subalgebra $\cA_{\cO'}$ of $\centr(L)$. Moreover, the next theorem can be proved with the same arguments as Theorem \ref{thm:goodearl}.

\begin{theorem}\label{thm-cA}
    Let $\cA$ be a commutative subring of $\kk [\partial]$ such that $L\in \cA$ with $\ord (L)=n$. Consider the set of orders $\cO_\cA = \{[\ord(B)]_n\ :\ B \in \cA\}$. Then, $\cO_\cA$ is a subgroup of~$\zz_n$ and any set $\{G_i\ :\ i \in \cO_\cA\}$, where $G_i$ is order minimal within $\cA$, is a $\Const[L]$-module basis of $\cA$.
\end{theorem}

    \begin{example}[Exponential]\label{exm:exponential:set}
        Let $L = \partial^3 +\frac{6}{\cosh(x)^2}\partial$. Using the algorithms in~\cite{JimenezPastor2025b}, we obtain that the centralizer of $L$ is of rank 1, and can be written as a $C[L]$-module generated by $\centr(L) = C[L]\oplus G_1C[L] \oplus G_2C[L]$, where operators $G_1$ and $G_2$ have orders $4$ and $5$, respectively, with the following formulas:
        \begin{align*}
            G_1 & = \partial^4 + \left(\frac{8}{\cosh(x)^2} - \frac{4}{3}\right)\partial^2 - \frac{8\sinh(x)}{\cosh(x)^3}\partial,\\
            G_2 & = \partial^5 + \frac{10}{\cosh(x)^2}\partial^3 \\ & - \frac{20 \sinh(x)}{\cosh(x)^3}\partial^2 + \left(\frac{16}{9} + \frac{80}{3\cosh(x)^2} - \frac{20}{\cosh(x)^4}\right)\partial.
        \end{align*}
    \end{example}

    \begin{example}[Elliptic]\label{exm:elliptic:set}
        Let $L = \partial^4 - 12\wp\partial^2 + 1$. Using the algorithms in~\cite{JimenezPastor2025b}, we obtain that the centralizer of $L$ is of rank 1 and can be written as
        \[\centr(L) = C[L] \oplus G_1C[L] \oplus G_2C[L] \oplus G_3C[L],\]
        where the operators $G_1,G_2$ and $G_3$ have orders $5,6$ and $7$, respectively, and have the following shape:
        \begin{align*}
            G_1 & = \partial^5 - 15\wp\partial^3 - \frac{15}{2}\wp'\partial^2-3g_2\partial,\\
            G_2 & = \partial^6 - 18\wp\partial^4 - 18\wp'\partial^3 - \left(36\wp^2 + 9g_2\right)\partial^2,\\
            G_3 & = \partial^7 - 21\wp\partial^5 - \frac{63}{2}\wp'\partial^4 - \left(126\wp^2+21g_2\right)\partial^3 - 63\wp\wp'\partial^2 - 27g_3\partial.
        \end{align*}
    The maximal subalgebra of rank $2$ is $\cA_2=C[L,G_2]=C[L]\oplus C[L] G_2$. The subalgebra
$$\cA_1=\Const[L,G_1] =  \Const[L]\oplus \Const[L] G_1 \oplus \Const[L] G_1^2\oplus \Const[L] G_1^3$$
has rank $1$ but is not maximal. It is written as a $\Const [L]$-module using its Goodearl's basis $\cG_1$, illustrating Theorem \ref{thm-cA}.
\end{example}

\section{Burchnall-Chaundy ideals}\label{subsec:bc}

Before proceeding to the main results of this paper, let us recall the definition of the Burchnall-Chaundy ideal. In their prominent results in 1923 and 1931~\citep{BC1,BC2}, Burchnall and Chaundy proved a correspondence between pairs of commuting differential operators and planar algebraic curves. 

For our purpose, we will focus on the ideas in~\cite{Rueda2025}. From the existence of the Burchnall-Chaundy polynomial, it is clear that the set of all polynomials that vanish when evaluated at a pair of commuting differential operators $L,P$ is an ideal. 
Moreover, we can consider a finite set of commuting differential operators $A_0,\ldots ,A_{t-1}$ in $\kk [\partial]$ to define the Burchnall-Chaundy ideal of a finitely generated commutative $\Const$-algebra in $\kk [\partial]$
\begin{equation*}
    \cA=\Const [A_0,A_1,\ldots ,A_{t-1}].
\end{equation*}
Given algebraic variables $Y=\{y_0,\ldots ,y_{t-1}\}$, $\partial(y_i)=0$, we establish the ring homomorphism $\phi_{\cA}$ from the polynomial ring $\Const[Y]$ to $\cA$ by $\phi_{\cA}(y_i)=A_i$ and $\phi_{\cA}(c)=c$. Thus we have an isomorphism of $\Const$-algebras
\begin{equation}\label{eq-cA}
    \cA\simeq \frac{\Const[Y]}{\ker(\phi_{\cA})}.
\end{equation}

\begin{definition}[\citep{Rueda2025}]
  Let $\cA$ be finitely generated commutative $\Const$-algebra in $\kk [\partial]$. The \emph{Burchnall-Chaundy (BC) ideal} of $\cA$ is $$\bc(\cA):=\ker(\phi_{\cA}),$$ the ideal of polynomials $p(Y)$ in $\Const[Y]$ such that $\phi_{\cA}(p)=0$, that we also write as $p(A_0,A_1,\ldots,A_{t-1}) = 0$. We will call \emph{BC polynomials} to the elements of $\bc(\cA)$.
\end{definition}

Coming back to the set up of this paper with a differential operator $L$ in $\kk[\partial]$, we can consider the BC ideal of the entire centralizer $\centr(L)$.  Let us fix a Goodearl's basis $\cG=\{1,G_1,\ldots ,G_{t-1}\}$ of $\centr (L)$ as a $C[L]$-module. We can write:
\begin{equation}
    \centr (L)=\Const[L,G_1,\ldots ,G_{t-1}]=\Const [L]\oplus \Const [L] G_1\ldots \oplus \Const [L] G_{t-1}.
\end{equation}
For $\cA=\centr (L)$, consider algebraic variables $\lambda$ and $\mu_1,\ldots,\mu_{t-1}$, to establish the ring homomorphism 
$\phi_{L}$ as the evaluation morphism from $\Const [\lambda, \overline{\mu}]$, that we will denote by 
\begin{equation}
 \phi_L:\Const [\lambda, \overline{\mu}]\rightarrow \centr (L), \,\,\, \phi_L(\lambda)=L,  \phi_L(G_i)=\mu_i.  
\end{equation}
\begin{definition}
    The \emph{Burchnall-Chaundy (BC) ideal} of $L$ is the BC ideal of its centralizer $\centr(L)$, which is the kernel of $\phi_L$ and we will denote by $\bc(L)$.
\end{definition}  


Burchnall-Chaundy ideals are prime ideals in the given ring $\Const [\lambda,\overline{\mu}]$ of constant coefficient polynomials because $\cA$ are integral domains, an immediate consequence of \eqref{eq-cA}. 

It was first shown by Burchnall and Chaundy~\cite{BC1, BC2} that BC-polynomials define algebraic curves. The quotient field of $\centr(L)$ is a function field in one variable, whose transcendence degree is one. Moreover any subalgebra $\cA$ of $\centr(L)$ has Krull dimension one, it is the affine ring of an algebraic curve $\Spec(\cA)$. This result was generalized later by Krichever~\cite{K78}, followed by Mumford~\cite{Mum2}, Verdier~\cite{Ve}, Mulase~\cite{Mulase} and~\cite[Theorem 1.8]{BZ} with coefficient in $\Const((x))$.
The domain $\Const[\Gamma] = \Const[\lambda,\overline{\mu}]/\bc(L)$ is the coordinate ring of the algebraic curve
  \begin{equation}
      \Gamma:=\{\eta\in\Const^n\mid g(\eta)=0, \forall g\in\bc (L)\},
  \end{equation}
  the \emph{spectral curve} of $L$. For $L$ of order $2$, it is known that the curve is planar. For higher orders, the curve may not be planar anymore~\citep{Rueda2024}.

It is important to remark that the definition of BC ideals does not depend on the choice of Goodearl's basis. The ideals $\bc (\cA)$ and the corresponding algebraic curves are always equivalent. Essentially, $\bc(L)$ encodes all the algebraic relations between $L$ and all the elements in the centralizer $\centr(L)$. Given $P \in \centr(L)$, using Theorem~\ref{thm:goodearl}, we know that there are unique polynomials $p_i\in \Const [\lambda]$ such that
\[P = p_0(L) + p_1(L)G_1 + \ldots + p_{t-1}(L)G_{t-1}.\]

\begin{definition}\label{def-coord}
Given $P \in \centr(L)$.
We can say that  $(p_0,\ldots ,p_{t-1})\in \Const [\lambda]^t$ are the \emph{coordinates of $P$ in the $\Const [L]$-basis $\cG$}.    
\end{definition}

We can see $P$ as the evaluation by $\phi_L$ of a polynomial in $\Const[\lambda,\overline{\mu}]$ that is linear in $\mu_1,\ldots,\mu_{t-1}$.

\begin{lemma}\label{lem:zero_check}
    Let $q(\lambda,\overline{\mu})\in \Const[\lambda,\overline{\mu}]$ be a linear polynomial in $\overline{\mu}$. Then $q \in \bc(L)$ if and only if $q = 0$.
\end{lemma}
\begin{proof}
Recall that $\bc(L)=\ker  (\phi_L)$.
    Assume that $q \neq 0$ and  
\[q(\lambda,\overline{\mu}) = q_0(\lambda) + q_1(\lambda)\mu_1 + \ldots + q_{t-1}(\lambda)\mu_{t-1},\,\, q_i\in \Const [\lambda].\]
    If $q_i(\lambda) \neq 0$, then the summand $q_i(\lambda)\mu_i$ evaluates by $\phi_L$ to an operator of order congruent with $\ord (G_i) \pmod{n}$. By  Theorem~\ref{thm:goodearl}, $\ord(\phi_L(q_i(\lambda)\mu_i))\neq \ord(\phi_L(q_j(\lambda)\mu_j))$ for $j\neq i$.  Hence, 
    \[\ord(\phi_L(q)) = \max\{\deg_\lambda(q_i)n+\ord(G_i)\ :\ i \in \{0,\ldots,t-1\}\} \geq 0,\]
    or equivalently, $\phi_L(q) \neq 0$.
\end{proof}


\section{Gröbner basis for Burchnall-Chaundy ideals}\label{sec:result}

In this section, we obtain the main result of this paper, namely, the computation of a Gr\" obner basis of the Burchnall-Chaundy ideal of an ordinary differential operator $L \in \kk[\partial]$. Assume that we have computed a Goodearl's basis of $\centr(L)$ as a $\Const [L]$-module $\cG=\{G_1,\ldots,G_{t-1}\}$.

The goal is to present a Gröbner basis of the ideal $\bc(L)$, with the special property that the normal form of a polynomial in $\Const[\lambda,\overline{\mu}]$ is linear in $\overline{\mu}=\{\mu_1,\ldots ,\mu_{t-1}\}$ with respect to such Gröbner basis. To do so, we define a weighted monomial order $<_w$ on $\Const[\overline{\mu}]$, by assigning weights $w(\mu_i):=\ord(G_i)$. If we denote a monomial in  $\Const[\overline{\mu}]$ by $\mu^{\alpha}$ with $\alpha\in \nn^{t-1}$, its weight is $w(\mu^{\alpha}):=\sum_{i=1}^{t-1} w(\mu_i) \alpha_i$. The weighted monomial ordering $<_w$ is defined after fixing a lexicographic ordering $<_{ord}$, setting $\mu_i<_{ord} \mu_j$ if $\ord(G_i)<\ord(G_j)$. More precisely,
monomials in $\overline{\mu}$ are compared by their weight breaking ties with the lexicographic ordering.
See~\cite{Cox1996} for the definitions of product, weight and lexicographic monomial orders.


\begin{definition}\label{def:order}
We define a product order $\<$ in $\Const[\lambda,\overline{\mu}]$ consisting of the weighted order $<_w$ on variables $\{\mu_1,\ldots,\mu_{t-1}\}$ and assuming $\mu_i > \lambda^k$, for all $k\in\nn$. Denoting monomials in $\Const[\lambda,\overline{\mu}]$ by $\lambda^k\mu^{\alpha}$ with $\alpha\in \nn^{t-1}$, $k\in\nn$, we obtain
\begin{equation}
\lambda^k\mu^{\alpha}\,\<\,\lambda^{\ell}\mu^{\beta} \Leftrightarrow \mu^{\alpha}<_w \mu^{\beta} \mbox{ or } \mu^{\alpha} = \mu^{\beta} \mbox{ and } k<\ell.
\end{equation}    
\end{definition}

Let us consider $G_i$ and $G_j$ in $\cG$, which may be identical. Since both $G_i,G_j \in \centr(L)$, it is clear that $G_iG_j = G_jG_i \in \centr(L)$.  We can write uniquely:
\[G_i G_j = G_j G_i = p_{i,j,0}(L) + p_{i,j,1}(L)G_1 + \ldots + p_{i,j,t-1}(L)G_{t-1},\]
for unique polynomials $p_{i,j,k}(\lambda) \in \Const[\lambda]$. 
Thus a $\Const [L]$-basis $\cG$ of $\centr (L)$ determines the following set of polynomials 
\begin{equation}\label{eq-Rij}
    R_{i,j}(\lambda,\overline{\mu}):=\mu_i\mu_j - p_{i,j,0}(\lambda) - \sum_{k=1}^{t-1} p_{i,j,k}(\lambda)\mu_k,
\end{equation}
for every $i,j \in \{1,\ldots,t-1\}$.
It is clear that $R_{i,j}(\lambda,\overline{\mu}) \in \bc(L)$ since 
$$\phi_{L}(R_{i,j})=R_{i,j}(L,G_1,\ldots,G_{t-1}) = 0.$$
These polynomials $R_{i,j}$ are a really strong set of relations that allow to reduce the total degree of polynomials in $\Const[\lambda,\overline{\mu}]$ with respect to the variables $\overline{\mu}$ to one. This is the main idea behind the following results.

\begin{proposition}\label{prop:groebner_basis}
  Let $L$ be a differential operator in $\kk [\partial]$ of order $n$. Every $\Const [L]$-basis $\cG$ of $\centr (L)$ determines a Gröbner basis
  \begin{equation}\label{eq-Gbasis}
    \cB_{\cG} := \{R_{i,j}(\lambda,\overline{\mu})\ :\ i \leq j \in \{1,\ldots,t-1\}\}     
  \end{equation}
of $(\cB_{\cG})$ with respect to the monomial order $\<$.
\end{proposition}
\begin{proof}
  Consider the ideal $I = \left(\cB_{\cG}\right)$ generated by $\cB_{\cG}$ in $\Const[\lambda,\overline{\mu}]$. By Definition~\ref{def:groebner}, $\cB_{\cG}$ is a Gröbner basis of $I$ if and only $\ell(I)=\ell(\cB_{\cG})$. Since $\cB_{\cG}\subset I$, we have $\ell(B_{\cG}) \subseteq \ell(I)$.
Given a nonzero $p \in I$,  there are two options:
  \begin{enumerate}
    \item If $\deg_{\overline{\mu}}(\ell(p))>1$, then there exists $r\in\cB_{\cG}$ such that $\ell(r)$ divides $\ell(p)$, since $\ell(\cB_{\cG})$ contains all possible degree two monomials in $\overline{\mu}$. Thus $\ell(p) \in \ell(\cB_{\cG})$.
    \item If $\deg_{\overline{\mu}}(\ell(p))\leq 1$ then $p$ is linear in $\mu_1,\ldots,\mu_{t-1}$ and by Lemma~\ref{lem:zero_check}, $p \notin \bc(L) \supset I$.
\end{enumerate}
Hence, $\ell(I) \subseteq \ell(\cB_{\cG})$, and therefore $\cB_{\cG}$ is a Gröbner basis of the ideal $I$ w.r.t. the given monomial ordering $\<$.
\end{proof}

Proposition~\ref{prop:groebner_basis} shows that the set of polynomials $\cB_{\cG}$ has a strong computational structure. Among other properties, we can now reduce any polynomial $p \in \Const[\lambda,\overline{\mu}]$ to a polynomial $\tilde{p}$ such that $\ell(\tilde{p}) < \ell(R_{i,j})$ for every $i\leq j \in \{0,\ldots,n-1\}$. In particular, the normal form w.r.t. $\cB_{\cG}$ of any polynomial $p \in \Const[\lambda,\overline{\mu}]$ is a linear polynomial in the variables  $\overline{\mu}$.

\begin{theorem}\label{thm:burchnall_chaundy}
 Let $L$ be a differential operator in $\kk [\partial]$ of order $n$ and let $\cG$ be a $\Const [L]$-basis of its centralizer $\centr(L)$. Then $\cB_{\cG}$ is a Gröbner basis of the BC-ideal $\bc(L)$ with respect to the monomial ordering $\<$.
\end{theorem}
\begin{proof}
  Let $\cB_{\cG}$ be as in \eqref{eq-Gbasis} and let $p \in \bc(L)$. Let us reduce $p$ w.r.t. $\cB_{\cG}$, obtaining 
  \[p = q + \tilde{p},\]
  where $q \in (\cB_{\cG})$ and $\tilde{p}$ is the normal form of $p$. In particular, since $(\cB_{\cG}) \subset \bc(L)$, we obtain that $\tilde{p} \in \bc(L)$. But we know that $\tilde{p}$ is linear in $\overline{\mu}$. Using Lemma~\ref{lem:zero_check} we conclude that $\tilde{p} = 0$ and, hence $p \in (\cB_{\cG})$.
\end{proof}

We present now the main idea for an Algorithm \ref{alg:compute_bc_ideal} that computes $\bc (L)$, by computing the polynomials $R_{i,j}$, for $i,j \in \{1,\ldots,t-1\}$, of the Gr\" obner basis $\cB_{\cG}$. 
It uses multivariate division in $\Const [\lambda,\overline{\mu}]$ w.r.t. the monomial order $\<$ of Definition \ref{def:order}, that is weighted on the variables $\overline{\mu}$ using the orders of the operators in the $\Const [L]$-basis $\cG$.

For a polynomial $p \in \Const[\lambda,\overline{\mu}]$, this division can be used to check whether $p\in \bc(L)$ or not.
We will compute a \emph{quotient} $q \in \bc(L)$ and its normal form $\tilde{p}$ such that $p=q+\tilde{p}$. 
Let $P = \phi_L(p)$. If $\ord(P) \equiv \ord(G_i)\pmod{n}$ and $\ord(P) \geq \ord(G_i)$, for some $i \in \{1,\ldots,t-1\}$, then we can write $\ord(P) = kn + \ord(G_i)$. Consider $r = p - \lc(p)\lambda^k\mu_i$. It is clear that $R = \phi_L(r) = P - \lc(P)L^kG_i$ and satisfies $\ord(R) < \ord(P)$.
More precisely, we can iterate this process. Set $q = 0$ and $\tilde{p} = 0$:
\begin{itemize}
    \item If $P = 0$: then $p \in \bc(L)$, hence set $q \gets q+p$, and terminate.
    \item If $\ord(P) \not\equiv \ord(G_i) \pmod{n}$ for all $i \in \{1,\ldots,t-1\}$: then $p \notin \bc(L)$, set $\tilde{p} \gets p$ and terminate.
    \item Assume $\ord(P) \equiv \ord(G_i) \pmod{n}$ for some $i \in \{1,\ldots,t-1\}$. There are two possibilities:
        \begin{itemize}
            \item $\ord(P) \geq \ord(G_i)$: compute $r$ and $R$ as above and set $q \gets q + \lc(p)\lambda^k\mu_i$, $p \gets r$, $P\gets R$ to start over.
            \item $\ord(P) < \ord(G_i)$: then $p \notin \bc(L)$, so set $\tilde{p} \gets p$ and terminate.
        \end{itemize}
\end{itemize}

The next algorithm, Algorithm~\ref{alg:reduce_as_module}, goes further and checks if an operator $P\in \kk [\partial]$ belongs to the centralizer $\centr (L)$ or not, computing at the same time the coordinates of an operator of the centralizer, as in Definition~\ref{def-coord}. Algorithm~\ref{alg:reduce_as_module} could also be used for any commutative $\Const$-algebra containing $L$ for which a $\Const[L]$-module basis has been computed as in Theorem \ref{thm-cA}.

\begin{algorithm}[!ht]
\caption{\texttt{Reduce\_as\_Module}}\label{alg:reduce_as_module}
    \Input{$P$ in $\kk[\partial]$; $L$ in $\kk[\partial]$ in normal form; Goodearl's basis $\cG=\{G_0,\ldots,G_{t-1}\}$ of $\centr(L)$ as a $\Const[L]$-module.}
    \Output{Coordinates $(p_0,\ldots,p_{t-1})$ of $P$ in $\cG$ if $P\in \centr(L)$, \emph{Error} otherwise.}

    $n \gets \ord(L)$\;
    $\mathcal{O} \gets \{\ord(G_i) \pmod{n}\ :\ i = 0,\ldots,t-1\}$\;
    $\mathcal{G} \gets \{\ord(G_i) \pmod{n} \mapsto G_i\ :\ i = 0,\ldots,t-1\}$\;
    $p_i \gets 0\ \text{for}\ i = 0,\ldots,t-1$\;

    \While{$\ord(P) \geq n$}{%
        $m_0 \gets \ord(P) \pmod{n}$\;
        \If{$m_0 \notin \mathcal{O}$}{%
          \Return \emph{Error}\;
        }
        $G \gets \mathcal{G}(m_0)$\;
        \If{$\ord(P) < \ord(G)$}{%
          \Return \emph{Error}\;
        }
        $k \gets \left(\ord(P) - \ord(G)\right)/n$\tcp*{Exact division}
        $p_{m_0} \gets p_{m_0} + \lc(P)\lambda^k$\;
        $P \gets P - \lc(P)L^kG$\tcp*{Reduces the order of $P$}
    }

    \If{$P \neq 0$}{%
        \Return \emph{Error}\;
    }
    \Return $\left(p_0,\ldots,p_{t-1}\right)$\;
\end{algorithm}

In Algorithm~\ref{alg:reduce_as_module}, step 2, we define the map $\mathcal{G}: \mathcal{O}\rightarrow \centr(L)$ that takes a congruence class $(\text{mod } {n})$ and returns the corresponding element in Goodearl's basis whose order is congruent with this element. 

\begin{example}[Continuation of Example~\ref{exm:exponential:set}]\label{exm:exponential:g1g2}
    Consider the operators $G_1$ and $G_2$ as in Example~\ref{exm:exponential:set} and their product $G_1G_2$. This is an operator of order $9=3\cdot3$, so the operator $G_1G_2 - L^3$ has order strictly smaller:
    \begin{align*}
        G_1G_2 - L^3 = & -\frac{4}{3}\partial^7 + \left(\frac{16}{9} - \frac{56}{3\cosh(x)^2}\right)\partial^5 + \frac{224\sinh(x)}{3\cosh(x)^3}\partial^4\\
                       & +\left(\frac{224}{\cosh(x)^4}-\frac{544}{3\cosh(x)^2}-\frac{64}{27}\right)\partial^3\\
                       & +\left(\frac{640}{3\cosh(x)^3}-\frac{448}{\cosh(x)^5}\right)\sinh(x)\partial^2\\
                       & +\left(\frac{512}{\cosh(x)^4}-\frac{896}{9\cosh(x)^2}-\frac{448}{\cosh(x)^6}\right)\partial
    \end{align*}

    To reduce this operator of order 7 we use that $7 = 3+4$, so we can reduce it with $\frac{4}{3} LG_1$. This leads to the operator
    \[G_1G_2 - L^3 + \frac{4}{3}LG_1 = \frac{-64}{27}\partial^3 - \frac{128}{9\cosh(x)^2}\partial = -\frac{64}{27}L.\]
    Hence, we conclude that
    \[G_1G_2 = L^3 - \frac{4}{3}LG_1 - \frac{64}{27}L.\]
\end{example}

Computing the polynomials $R_{i,j}(\lambda,\overline{\mu})$ for $i\leq j \in \{1,\ldots,t-1\}$ is now straightforward. We can do so by computing the products $G_i G_j$ and then reducing them as linear combinations of the operators $G_0,\ldots,G_{t-1}$ and $L$. The algorithm is described in Algorithm~\ref{alg:compute_bc_ideal}.

\begin{algorithm}[!ht]
\caption{\texttt{BC\_Ideal}}\label{alg:compute_bc_ideal}
    \Input{$L$ in $\kk[\partial]$ in normal form; Goodearl's basis $\mathcal{G}=\{1,G_1,\ldots,G_{t-1}\}$ for $\centr(L)$ as a $\Const[L]$-module.}
    \Output{A Gröbner basis of the Burchnall-Chaundy ideal $\bc(L)$ w.r.t the monomial order $\<$.}

    \For{$i=1,\ldots,t-1$}{%
        \For{$j=i,\ldots,t-1$}{%
            $\left(p_0,\ldots,p_{t-1}\right) \gets \texttt{Reduce\_as\_Module}(G_i\cdot G_j, L, \mathcal{G})$\;
            
            $R_{i,j} \gets \mu_i\mu_j - \sum_{l=0}^{t-1} p_l(\lambda)\mu_l$\;
        }
    }
    \Return $\{R_{i,j}(\lambda,\overline{\mu})\ :\ i\leq j \in \{1,\ldots,t-1\}\}$\;
\end{algorithm}

\begin{example}[Continuation of Example~\ref{exm:exponential:g1g2}]
    We use the algorithm \texttt{BC\_ideal} to compute the Gr\" obner basis of $\bc (L)$, for the operator $L$ as in Example~\ref{exm:elliptic:set}. We need to compute the coordinates in the $\Const [L]$-basis $\cG=\{1,G_1,G_2\}$ of each of the operators $G_1^2, G_1G_2$ and $G_2^2$ using the algorithm  \texttt{Reduce\_as\_Module}. Then translate these computations into polynomials using the variables $\lambda,\mu_1$ and $\mu_2$.

    $G_1^2$ is an operator of order $8=3+5$, so we start by reducing with $LG_2$, leading to an operator of order $6$ which coincides with $\frac{8}{3}L^2$. Hence we obtain the polynomial $R_{1,1} = \mu_1^2 - \lambda\mu_2 + \frac{8}{3}\lambda^2$.

    For $G_1G_2$, we already shown the details in Example~\ref{exm:exponential:g1g2}, leading to the polynomial $R_{1,2} = \mu_1\mu_2 + \frac{4}{3}\lambda\mu_1 + \left(\frac{64}{27}\lambda - \lambda^3\right)$.

    Finally, we compute the operator $G_2^2$ which is reduced sequentially by the operators $L^2G_1$ (order 10), $LG_2$ (order 8), $L^2$ (order 6) and $G_1$ (order 4), obtaining the polynomial $R_{2,2} = \mu_2^2 + \left(\frac{64}{27}\lambda - \lambda^2\right)\mu_1 - \frac{4}{3}\lambda\mu_2 - \frac{32}{9}\lambda^2$. Putting everything together we get
    \[\bc(L) = \left(R_{1,1},R_{1,2},R_{2,2}\right).\]
    Since $L$ has order $3$, the BC-ideal can also be computed using differential resultants as in ~\cite{Rueda2024}. More precisely $\bc(L)=(f_1,f_2,f_3)$ where
    $f_i=\dres(L-\lambda, G_i-\mu_i)$, $i=1,2$ and $f_3=\dres(G_1-\mu_1, G_2-\mu_2)$.
\end{example}

\begin{example}[Continuation of Example~\ref{exm:elliptic:set}]
We use the algorithm \texttt{BC\_ideal} to compute the Gr\" obner basis of $\bc (L)$, for the operator $L$ as in Example~\ref{exm:elliptic:set}. We need to compute the coordinates in the $\Const [L]$-basis $\cG=\{1,G_1,G_2,G_3\}$, using Algorithm~\ref{alg:reduce_as_module}, of each of the operators \[G_1^2,\quad G_1G_2,\quad G_1G_3,\quad G_2^2,\quad G_2G_3,\quad G_3^2.\]

We obtain the Gröbner basis of $\bc(L)$ determined by $\cG$, of the form $\bc(L) = \left(R_{1,1},R_{1,2},R_{1,3},R_{2,2},R_{2,3},R_{3,3}\right)$, where
\begin{align*}
    R_{1,1} = &\mu_1^2 +\left(\frac{3 g_2}{4} +1 - \lambda\right)\mu_2 + \frac{27g_3}{4}(\lambda -1),\\
    R_{1,2} = &\mu_1\mu_2 + (1 -\lambda)\mu_3,\\
    R_{1,3} = &\mu_1\mu_3 + \frac{27 g_3}{4}\mu_2 + \left(\frac{9g_2^2+15g_2+4}{4} \right.\\
    & \left.- \frac{9g_2^2+30g_2+12}{4}\lambda + \frac{15g_2 + 12}{4}\lambda^2 - \lambda^3\right),\\
    R_{2,2} = &\mu_2^2 +\left(1+3g_2\right)-\left(6g_2+3\right)\lambda + \left(3g_2+3\right)\lambda^2 - \lambda^3,\\
    R_{2,3} = &\mu_2\mu_3 - \left(\lambda^2 - \left(3g_2+2\right)\lambda + \left(3g_2+1\right)\right)\mu_1,\\
    R_{3,3} = &\mu_3^2 - \left(\lambda^2 - \frac{15g_2 + 8}{4}\lambda + \frac{9g_2^2+15g_2+4}{4}\right)\mu_2 \\
    & + \left(\frac{27g_3}{4}\lambda^2-\frac{(81g_2 + 54)g_3}{4}\lambda+\frac{(81g_2 + 27)g_3}{4}\right).
\end{align*}
Observe that in this case, $\centr (L)$ has a $\Const [L]$-submodule of rank $2$, namely $\cA_2=\Const [L,G_2]=\Const [L]\oplus \Const [L] G_2$ since $\ord(G_2)\equiv 2 \pmod{4}$ with $\Const [L]$-basis $\cG_2=\{1,G_2\}$. In this case $\bc (\cA_2)=(R_{2,2})$ the Gr\" obner basis is just $\{R_{2,2}\}$ and $R_{2,2}$ is exactly the differential resultant of $L-\lambda$ and $G_2-\mu_2$. 

To apply the algorithm to other commutative subalgebras, for instance $\cA_1=\Const [L,G_1]$ (or $\Const [L,G_3]$), we use the Goodearl's basis $\cG_1$ to define the homomorphism $\phi_{\cA_1}$ sending $\lambda\mapsto L$, $\mu_1\mapsto G_1$, $\mu_2\mapsto G_1^2$ and $\mu_3\mapsto G_1^3$.
then applying the algorithm \texttt{BC\_ideal} to obtain
the Gr\" obner basis of the BC ideal $\bc(\cA_1)$.  
In this case it is much cleaner to just assign $\mu\mapsto G_1$ and obtain $\Const [L,G_1]$ isomorphic to $\Const [\lambda,\mu]/(f)$, where $f$ is the differential resultant of $L-\lambda$ and $G_1-\mu$. 

\end{example}

The BC-ideal $\bc (L)$ contains the ideal generated by the radicals of the differential resultants of all pairs in $\{L-\lambda, G_i-\mu_i:i=1,\ldots ,t-1\}$. For $\ord(L)>3$ these ideals do not coincide in general. 

\section{Primality of Burchnall-Chaundy ideals}\label{sec:primality}

In order to study the differential Galois theory of the spectral problem $(L-\lambda)(y)=0$, 
under the assumption that $\centr (L)$ is non-trivial, $\centr (L)\neq \Const[L]$,
we must consider a field of coefficients containing $\kk$ and the spectral parameter $\lambda$, but reflecting the fact that $\lambda$ is not a free parameter. It is instead governed by the spectral curve $\Gamma$, see \cite{BEG, MRZ2, Rueda2024, Rueda2025}.

In Section~\ref{subsec:bc}, we noted that for any differential operator $L \in \mathbb{K}[\partial]$ with non-trivial centralizer, the Burchnall-Chaundy ideal $\bc(L) \subset \Const[\lambda,\overline{\mu}]$ is always a prime ideal since the quotient ring  
\begin{equation}
    \Const [\Gamma]:=\frac{\Const [\lambda,\overline{\mu}] }{\bc (L)}
\end{equation}
is an integral domain isomorphic to $\centr (L)$.

Let us consider the differential ideal $[\bc(L)]$ generated by $\bc(L)$ in the polynomial ring $\kk[\lambda,\overline{\mu}]$. Recall that $\bc(L)$ is generated by constant coefficient polynomials $\bc(L)=(\cB_{\cG})$, being generated by the Gr\" obner basis $\cB_{\cG}=\{R_{i,j}\}$ determined by a given $\Const [L]$-basis of $\centr (L)$. 
Observe that $[\bc (L)]$ is in fact the ideal generated by $\{R_{i,j}\}$ in $\kk[\lambda,\overline{\mu}]$
\begin{equation}
    [\bc(L)] = \left(R_{i,j}\ :\ i,j \in \{1,\ldots,t-1\}\right).
\end{equation}
We can thus consider the quotient ring 
\begin{equation}
    \kk [\Gamma]:=\frac{\kk [\lambda,\overline{\mu}] }{[\bc (L)]},
\end{equation}
which contains $L-\lambda$, where we continue denoting $\lambda$ now the class $\lambda+[\bc (L)]$. Furthermore $\kk [\Gamma]$ is a differential ring, with extended derivation $\tilde{\partial}$ defined by
\begin{equation}
    \tilde{\partial} (f+[\bc (L)])=\partial (f)+[\bc (L)].
\end{equation}

In this section, we use the fact that $\cB_{\cG}$ is a Gr\" obner basis of $\bc (L)$ for the monomial order $\<$ established on monomials in the variables $\lambda,\mu_1,\ldots ,\mu_{t-1}$, see Definition \ref{def:order}, to prove that the extended ideal $[\bc(L)]$ is also a prime differential ideal. Although intuitive, this result cannot be proven as for $\bc(L)$: the homomorphism $\phi_L$ cannot be extended to the differential ring $\kk[\lambda,\overline{\mu}]$ because the non-constant elements of $\kk$ do not commute with $\partial$.

\begin{proposition}
Given $L\in \kk [\partial]$ and a $\Const [L]$-basis $\cG$ of $\centr (L)$. The set $\cB_{\cG}=\{R_{i,j}\}$ is a Gr\" obner basis of $[\bc (L)]$ w.r.t. the monomial order $\<$ in $\kk[\lambda,\overline{\mu}]$.
\end{proposition}
\begin{proof}
    We start by noting that, since $\cB_{\cG}$ is a Gröbner basis of $\bc (L)$ w.r.t. $\<$, within $\Const[\lambda,\overline{\mu}]$, then all S-polynomials reduce to zero \cite{Cox1996}. This fact does not change when we consider $\cB_{\cG}$ as a set of generators for the extended ideal $[\bc(L)]$. Hence, $\cB_{\cG}$ is also a Gr\" obner basis for the extended ideal $[\bc(L)]$.
\end{proof}

We will prove next that the differential ring $\kk [\Gamma]$ is an integral domain, or equivalently the primality of the differential ideal $[\bc (L)]$. 
For this purpose, we will extend the weight function $w$, given before Definition~\ref{def:order}, to monomials $\cM$ in the set of variables $\{\lambda\}\cup \overline{\mu}$, defining $w(M):=\ord(\phi_L(M))$, for every $M\in \cM$. Thus
$w(\lambda) = \ord(L)$ and $w(\mu_i) = \ord(G_i)$, $i=1,\ldots , t-1$.

\begin{lemma}\label{lem-monomial}
    For every positive weight $W \in \nn$, there is at most one monomial $T_W$ in $\cM$ linear in $\overline{\mu}$ and with weight $W$.
\end{lemma}
\begin{proof}
Assume $\ord (L)=n$.
    If $W \equiv 0 \pmod{n}$, then only the monomial $\lambda^{W/n}$ has weight $W$. Otherwise, given the $\Const [L]$-basis $\cG=\{G_0,G_1,\ldots ,G_{t-1}\}$, by Theorem~\ref{thm:goodearl}, we know that $\ord(G_i) \not\equiv \ord(G_j) \pmod{n}$ for $i \neq j$. Hence, it can only exist one index $i \in \{1,\ldots,t-1\}$ with $\ord(G_i) \equiv W \pmod{n}$. We can write
    \[W = a\ord(L) + \ord(G_i).\]
    If $a \geq 0$, then $\lambda^a \mu_i$ has weight $W$. If $a < 0$, then there is no monomial with weight $W$.
\end{proof}

Recall that the polynomials of $\cB_{\cG}$ have a very specific shape
\begin{equation}
    R_{i,j}=\mu_i \mu_j+ L_{i,j}
\end{equation}
where $L_{i,j}$  is linear in $\overline{\mu}$.
Due to the monomial order $\<$, reducing any polynomial by $\cB_{\cG}$ is equivalent to replacing (in any order) the products $\mu_i\mu_j$ by the corresponding polynomial $L_{i,j}$. By construction, the polynomial $L_{i,j}$ has exactly the same weight as $\mu_i\mu_j$, so when performing the substitution we can only influence terms of lower or equal weight than $\mu_i\mu_j$.

Given a polynomial $f$ in $\kk [\lambda, \overline{\mu}]$, 
its normal form $\tilde{f}$ w.r.t. the Gr\" obner basis $\cB_{\cG}$ will be a linear polynomial in the variables $\overline{\mu}$. Thus every class $f+[\bc (L)]$ of the quotient ring $\kk [\Gamma]$ has a representative $\tilde{f}$ that is a linear polynomial in $\overline{\mu}$. 
Lemma \ref{lem-monomial} is particularly useful because it implies that the highest weight term of $\tilde{f}$ is unique. So when we multiply two of these representatives, there is a unique term in the product with highest weight.

\begin{theorem}\label{thm:prime}
    The quotient ring $\kk[\Gamma]$ is a differential domain.
\end{theorem}
\begin{proof}
    Let us consider two non-zero elements in the quotient ring $\kk[\Gamma]$  
    \begin{equation}
        F_1=f_1+[\bc(L)],\,\,\, F_2=f_2+[\bc(L)],
    \end{equation}
    with representatives
    $f_1,f_2\in \kk[\lambda,\overline{\mu}]$ that are already reduced w.r.t. $[\bc(L)]$. We know that $f_1,f_2$ are linear polynomials in  $\overline{\mu}$, so their highest weight terms are of the form $\lambda^{a_1}\mu_{i_1}$ and $\lambda^{a_2}\mu_{i_2}$. We allow $i_1 = 0$ or $i_2 = 0$, considering $\mu_0 = 1$ when this term is a monomial in $\lambda$.

    The product $F_1F_2$ has representative $f_1 f_2$ whose highest term is of the form $\lambda^{a_1+a_2}\mu_{i_1}\mu_{i_2}$. If $i_1 = 0$ or $i_2 = 0$, then there is no reduction that can cancel this term, meaning that $F_1F_2 \neq 0$. If both $i_1$ and $i_2$ are distinct from zero, then we will replace $\mu_{i_1}\mu_{i_2}$ by $ L_{i_1,i_2}$ creating a highest weight term of the same weight as $\lambda^{a_1+a_2}\mu_{i_1}\mu_{i_2}$ but linear in $\overline{\mu}$. This new term cannot be canceled by any previous nor further substitutions, hence $f_1f_2$ does not reduce to the zero polynomial and $F_1F_2 \neq 0$.

    This shows that the product of two non-zero elements in the quotient is always non-zero, so $\kk[\Gamma]$ is an integral domain.
\end{proof}

\begin{corollary}
    The differential ideal $[\bc (L)]$ is a prime ideal in $\kk [\lambda,\overline{\mu}]$. 
\end{corollary}

We can now consider the fraction field $\kk (\Gamma)$ of the domain $\kk [\Gamma]$ with the natural extended derivation $\tilde{\partial}$. This is the new differential field of coefficients $(\kk (\Gamma),\tilde{\partial})$ to study the spectral problem $(L-\lambda)(y)=0$. 

For a Schr\" odinger operators, ordinary differential operators of order $2$ in normal form, \emph{spectral Picard-Vessiot fields} were defined in \cite{MRZ2}. These are differential extensions of $\kk (\Gamma)$ by a fundamental system of solutions of $(L-\lambda)(y)=0$, without extending the field of constants. In the case of Schr\" odinger operators, the field of constants of $\kk (\Gamma)$ was shown to be $\Const (\Gamma)$.  
We are now ready to study  Picard-Vessiot extensions for algebro-geometric ODOs of arbitrary order. 

\section{Conclusions and further work}\label{sec:problems}

In this paper we studied the structure of Burchnall-Chaundy ideals, that define the relations between the elements in the centralizer $\centr(L)$ of an ordinary differential operator $L$, with coefficients in an arbitrary differential field $\kk$. Our results are based on the $\Const[L]$-module structure of centralizers, which was previously studied by K. Goodearl ~\cite{Goodearl1983} and many other authors, and the recently developed symbolic algorithms for their computation ~\cite{JimenezPastor2025b}. 

We designed an algorithm to compute the set of generators of a BC-ideal in a polynomial ring $\Const [\lambda, \overline{\mu}]$ with constant coefficients.  Moreover, we showed that this set of generators is a Gr\" obner basis, for a product ordering $\<$ on monomials in the variables $\lambda, \overline{\mu}$, after assigning weights to the chosen variables $\overline{\mu}$. In this manner, we achieve the effective computation of the spectral curve $\Spec (\cA)$ of a commutative ring of ODOs and, in particular, the spectral curve $\Gamma$ of the centralizer $\centr(L)$. 

Furthermore, by means of these Gr\" obner bases, the normal form of a polynomial with respect to the BC-ideal is a linear polynomial in the set of variables $\overline{\mu}$.  This allows directly proving important results to develop the Picard-Vessiot theory for spectral problems $(L-\lambda)(y)=0$. We showed in this paper that the prime ideal $\bc (L)$ generates a  prime differential ideal $[\bc(L)]$ in $\kk[\lambda, \overline{\mu}]$, allowing to consider a differential domain $\kk [\Gamma]$ and its fraction field $\kk (\Gamma)$ that contain $L-\lambda$. Over this new field of coefficients $\kk (\Gamma)$, we will be able to compute right factors of $L-\lambda$ and develop effective algorithms to compute solutions of $(L-\lambda)(y)=0$, to compute the famous Baker-Akheizer function \cite{K78, Schilling}. 

These results are not only theoretical, the algorithms are implemented in the open-source package \texttt{dalgebra} of SageMath. This package is a software under active development that can be freely accessed from 
\begin{center}\url{https://github.com/Antonio-JP/dalgebra}\end{center}
In version \texttt{0.0.8}, we added a Jupyter notebook within the repository that includes all the examples from this paper with the extended computations that can be seen as a showcase on how to use the software. The notebook can be accessed from
\begin{center}{\url{https://github.com/Antonio-JP/dalgebra/blob/ISSAC-2026/notebooks/ISSAC26_PaperExamples.ipynb}}\end{center}

\section*{Acknowledgements}
All authors are partially supported by the grant PID2021-124473NB-I00, “Algorithmic Differential Algebra and Integrability” (ADAI) from the Spanish MCIN/AEI /10.13039/501100011033 and by FEDER, UE. 

The authors would like to thank Carlos D'Andrea and Carlos Arreche for the very insighful discussions on these topics during their visits to Madrid.
\vspace{-1em}

\bibliographystyle{abbrv}
\bibliography{main}

\end{document}